\documentclass[reqno,11pt]{amsart}
\usepackage{amsmath, latexsym, amsfonts, amssymb, amsthm, amscd,mathrsfs}
\usepackage{graphics,epsf,psfrag}
\numberwithin{equation}{section}

\newtheorem{maintheorem}{Theorem}
\newtheorem{theorem}{Theorem}[section]
\newtheorem{fact}[theorem]{Fact}
\newtheorem{con}[theorem]{Conjecture}
\newtheorem{lemma}[theorem]{Lemma}
\newtheorem{proposition}[theorem]{Proposition}

\newtheorem{remark}[theorem]{Remark}
\newtheorem{definition}[theorem]{Definition}

\newtheorem{question}[theorem]{Question}

\newcommand{\R}{\mathbb{R}}
\newcommand{\Z}{\mathbb{Z}}

\renewcommand{\tilde}{\widetilde}

\newcommand{\cT}{{\ensuremath{\mathcal T}} }

\DeclareMathSymbol{\leqslant}{\mathalpha}{AMSa}{"36} 
\DeclareMathSymbol{\geqslant}{\mathalpha}{AMSa}{"3E} 
\DeclareMathSymbol{\eset}{\mathalpha}{AMSb}{"3F}     
 
\newcommand{\gl}{\lambda}

\renewcommand{\epsilon}{\varepsilon}

\makeatletter
\def\captionfont@{\footnotesize}
\def\captionheadfont@{\scshape}

\long\def\@makecaption#1#2{%
  \vspace{2mm}
  \setbox\@tempboxa\vbox{\color@setgroup
    \advance\hsize-6pc\noindent
    \captionfont@\captionheadfont@#1\@xp\@ifnotempty\@xp
        {\@cdr#2\@nil}{.\captionfont@\upshape\enspace#2}%
    \unskip\kern-6pc\par
    \global\setbox\@ne\lastbox\color@endgroup}%
  \ifhbox\@ne 
    \setbox\@ne\hbox{\unhbox\@ne\unskip\unskip\unpenalty\unkern}%
  \fi
  \ifdim\wd\@tempboxa=\z@ 
    \setbox\@ne\hbox to\columnwidth{\hss\kern-6pc\box\@ne\hss}%
  \else 
    \setbox\@ne\vbox{\unvbox\@tempboxa\parskip\z@skip
        \noindent\unhbox\@ne\advance\hsize-6pc\par}%
\fi
  \ifnum\@tempcnta<64 
    \addvspace\abovecaptionskip
    \moveright 3pc\box\@ne
  \else 
    \moveright 3pc\box\@ne
    \nobreak
    \vskip\belowcaptionskip
  \fi
\relax
}
\makeatother
\def\writefig#1 #2 #3 {\rlap{\kern #1 truecm
\raise #2 truecm \hbox{#3}}}

\renewcommand{\Pr}{ \mathrm P}

\newcommand{ \TV}{ \mathrm{TV} }

\newcommand{ \cL}{ \mathcal L }

\usepackage[normalem]{ulem}

\begin{document}

\title{On sensitivity of uniform mixing times}
\author{Jonathan Hermon}
\thanks{
University of Cambridge, Cambridge, UK. E-mail: {\tt jonathan.hermon@statslab.cam.ac.uk}. Financial support by
the EPSRC grant EP/L018896/1.}

\date{}

\begin{abstract}
We show that  the order of the $L_{\infty}$-mixing time of simple random walks on a sequence of uniformly bounded degree graphs of size $n$ may increase by an optimal factor of $\Theta( \log \log n)$ as a result of a bounded perturbation of the edge weights. This answers a question and a conjecture  of Kozma.
\end{abstract}

\maketitle

\paragraph*{\bf Keywords:}
{\small Sensitivity; mixing-time; sensitivity of mixing times; hitting times.

\paragraph*{\bf MSC class:}
{\small 60J10.

\section{Introduction}

An important question is whether mixing times are robust under small changes to the geometry of the Markov chain. For instance, can bounded perturbations of the edge weights change the mixing time by more than a constant factor? Similarly,  how far apart can the mixing times of lazy simple random walks on two roughly-isometric graphs of bounded degree be? A related question is whether mixing times can be characterized up to universal constants (perhaps only under reversibility) using geometric quantities or extremal characterizations which are robust.\footnote{That is, using quantities which can change by at most some bounded factor under a bounded perturbation of the edge weights, or for lazy simple random walk on a bounded degree graph, under  rough-isometries.} Different variants of this question were asked by various authors such as Pittet and Saloff-Coste \cite{cf:Pittet}, Kozma \cite[p.\ 4]{cf:Kozma}, Diaconis and Saloff-Coste \cite[p.\ 720]{diaconis} and Aldous and Fill \cite[Open Problem 8.23]{cf:Aldous} (the last two references ask for an extremal characterization of the $L_{\infty}$-mixing time in terms of the Dirichlet form). 

Denote the $L_p$ mixing time of lazy simple random walk on a finite connected simple graph $G$ by $\tau_p(G)$ (see \eqref{eq: tau_p}). Kozma \cite{cf:Kozma} made the following conjecture:
\begin{con}[\cite{cf:Kozma}]
\label{con: Kozma}
Let $G$ and $H$ be two finite $K$-roughly isometric graphs (see Definition \ref{def: RI}) of maximal degree $\le d$. Then for some $C(K,d)$ depending only on $(K,d)$,
$$\tau_{\infty}(G) \le C(K,d) \tau_{\infty}(H).$$ 
\end{con}

Our main result, Theorem \ref{thm: 1}, asserts that this conjecture is false. We shall only consider the following particularly simple type of rough isometries. Let $G_{1}:=(V_{1},E_{1})$ be some graph. Let $G_{2}=(V_{2},E_{2})$ be a graph obtained from $G_{1}$ by ``\textbf{\emph{stretching}}" some of the edges of $G_{1}$ by a factor of at most $K$ (we say that $G_{2}$ is a $K$-\emph{stretch} of $G_{1}$). That is, for some $E \subset E_{1}$ we replace each edge $\{u,v\} \in E $ by a path of length at most $K$ (whose end-points are still denoted by $u$ and $v$). Note that $V_{1} \subset V_{2}$. The identity map is a $K$-rough isometry of $G_1$ and $G_2$.

\begin{maintheorem}
\label{thm: 1}
There exist two families of uniformly bounded degree simple connected graphs $(G_n=(V_n,E_n))_{n \ge 0}$ and $(G'_n)_{n \ge 0}$ and some $c>0$ such that $|V_n| \to \infty $ and  for each $n$, $G_n'$ is a $2$-stretch of   $G_n$  and
\begin{equation}
\label{eq: 1}
\tau_{1}(G_{n}') \ge c \tau_{\infty}(G_{n})\log \log |V_n|.
\end{equation}
\end{maintheorem}
\begin{remark}
\label{rem: opt}
The $\log \log |V_n| $ term in \eqref{eq: 1} is optimal. This follows from the fact that the Log-Sobolev constant, and as shown in \cite{cf:Kozma}, also the spectral-profile, provide bounds on $\tau_{\infty}(G)$ for a graph $G=(V,E)$ of maximal degree $d$ which are sharp up to a $C_d \log \log |V| $  factor. These bounds are \emph{robust} (i.e.~invariant, up to a constant factor) under rough isometries in the bounded degree setup (\cite[Lemmas 3.3 and 3.4]{diaconis}). 
\end{remark}
\begin{remark}
\label{rem: spectral}
Denote the upper bound on  $\tau_{\infty}(G)$ given by the spectral profile \cite[Theorem 1.1]{cf:Spectral} by $\rho(G)$ (see \eqref{eq: SPbound}). Kozma \cite{cf:Kozma} asked whether it is possible for a sequence of bounded degree graphs $H_n=(V(H_n),E(H_n))$ with $|V_n| \to \infty $ to satisfy $\rho(H_n)/\tau_{\infty}(H_{n}) \ge c\log \log |V(H_n)|$ for some absolute constant $c>0$. Until now, there was no known example exhibiting this behavior (or even one in which $\rho(H_n)/\tau_{\infty}(H_{n})$ diverges). Since in the bounded degree setup the spectral profile is robust (i.e.~invariant, up to a constant factor) under rough isometries, 
for the graphs from Theorem \ref{thm: 1} it must be the case that for all $n$ \[ \rho(G_n) \ge c_1 \rho(G_n') \ge c_1 \tau_{\infty}(G_{n}')  \ge c_2 \tau_{\infty}(G_{n})\log \log |V_n|. \]

In other words, to prove Theorem \ref{thm: 1} one must first construct a sequence of bounded degree graphs for which the spectral profile bound overshoots the order of the $L_{\infty} $ mixing time by an optimal factor. It follows from the analysis in \cite{cf:Kozma} that such example must have $\Theta( \log \log |V_n|) $ disjoint sets, whose stationary probabilities are of different logarithmic scales, so that each of which makes roughly the same contribution to $\rho(G_n) $.

 We believe that in general, for such example, the aforementioned sets can be chosen so that: (1) the walk can  visit only one (or at most some constant number) of them before it is mixed in $L_{\infty}$, and (2) after stretching some of the edges by a bounded factor, the walk has to pass through at least some fraction of these sets in order to mix in total-variation.    
\end{remark}
We say that a family of graphs is \textbf{\emph{robust}} if for every $C>0$ there exists some $K>0$ such that if we multiply the edge weights of some of the edges by a factor of at most $C$ on these graphs, the corresponding $L_{\infty}$ mixing times are preserved up to a factor of $K$. In \cite{cf:Ding} Ding and Peres studied robustness of $L_1$ mixing (see \S~\ref{s: related}). We note that the $L_{\infty}$ case is much harder (see the discussion in \S~\ref{s: related}).
\begin{remark}
Loosely speaking, stretching an edge by a factor of $K$ has the same effect as decreasing its weight to $1/K$. Indeed, in Theorem \ref{thm: 1} instead of considering a $2$-stretch of $G_n$ we could have considered a bounded perturbation of the edge weights (i.e.~the same example works for both setups). 
\end{remark}

There are numerous works aiming at sharp geometric bounds on the $L_{\infty}$-mixing time, $\tau_{\infty}$, such as Morris and Peres' evolving sets bound \cite{cf:Evolving}, expressed in terms of the expansion profile. The sharpest geometric bounds on $\tau_{\infty}$ are given in terms of the Log-Sobolev constant (see \cite{diaconis} for a survey on the topic) and the spectral profile bound, due to Goel et al.~\cite{cf:Spectral}. Both of which determine $\tau_{\infty}$ up to a multiplicative factor of order $\log \log [\max(e^e, 1/\min \pi (x))] $, where throughout $\pi$ shall denote the stationary distribution. 

\medskip
 
These type of geometric bounds on mixing-times are robust under bounded perturbations (and in the bounded degree setup, also under rough isometries). That is, changing some of the edge weights by at most some multiplicative constant factor can change these geometric bounds only by a constant factor. Theorem \ref{thm: 1} serves as a cautionary
note on the possibility of developing sharp geometric bounds on mixing times.

\medskip

 In contrast with Theorem \ref{thm: 1}, many well-known families of graphs are robust. Robustness of the $L_{\infty}$ and $L_1$ mixing times for
general (weighted) trees under bounded perturbation of the edge weights was established in \cite{cf:L2} (by Peres and the author) and \cite{PS} (by Peres and Sousi), respectively. Robustness of  $L_1$ mixing times for
general trees under rough-isometries was recently established in \cite{berry} by Addario-Berry and Roberts.  Some other examples, which we borrow from \cite{cf:Ding} (where robustness of the $L_1$ mixing time is considered, however apart from example (2) below, the analysis of the $L_{\infty}$ mixing time  is identical) are collected in the next proposition.
\begin{proposition}
\label{prop: robustness}
The following families of graphs are robust: (1) Tori $\{\Z_n^d:n \ge 1 \}$ for every fixed $d$; (2) The giant component of a supercritical Erd\H os-R\' enyei random graph $\mathcal{G}(n,c/n)$ for a fixed $c>1$; (3) Maximal connected component of a critical  Erd\H os-R\' enyei random graph $\mathcal{G}(n,c/n)$;    (4) The hypercube $\{0,1 \}^n$. 
\end{proposition}

\subsection{Definitions}

Given a (weighted) network $(V, E,(c_{e})_{e
\in E})$, where each edge $\{u,v\} \in E$ is endowed with a conductance (weight) $c_{u, v}=c_{v,u}>0$ (with the convention that $c_{u,v}=0$ if $\{u,v\} \notin E$),  a lazy random walk on $G=(V,E)$, $(X_t)$, repeatedly does the following: when the current state is $v\in V$,  the random walk will stay at $v$ with probability $1/2$ and move to vertex $u$ (such that $\{u,v\} \in E$) with probability $c_{u, v}/(2c_{v})$, where $c_v:=\sum_{w} c_{v, w}$.
The default choice for $c_{u, v}$ is 1 (in which case, we say that the random walk is {\em unweighted}), which corresponds to lazy simple random walk  on $G$ (in which at each step the walk with equal probability either stays put or moves to a new vertex, chosen from the uniform distribution over the neighbors of its current position). Its stationary distribution is given by $\pi(x):=c_x/c_V$, where $c_{V}:=\sum_{v \in V}c_v=2\sum_{e \in E}c_e$. This is a \emph{reversible} Markov chain, i.e.~$\pi(u) P(u,v)=\pi(v)P(v,u)$ for all $u,v \in V$, where throughout $P$ denotes the transition matrix of the walk.

\medskip

We denote by $\Pr_{x}^t$ (resp.~$\Pr_{x}$) the distribution of $X_t$ (resp.~$(X_t)_{t \ge 0 }$), given that the initial state is $x$.  
Let $\mu,\nu$ be two distributions on the state space $\Omega$. Denote $a_{\mu,\nu,\pi}:=\frac{|\mu (x)-\nu (x)|}{\pi (x)} $. 
  The family of $L_p$ distances   is 
defined as follows.
\begin{equation}
\label{eq: Lpdef}
\|\mu-\nu \|_{p,\pi}:=\begin{cases}\left(\sum_{x} \pi(x) a_{\mu,\nu,\pi}^p  \right)^{1/p}, & 1 \le p<\infty, \\
\max_{x\in \Omega} a_{\mu,\nu,\pi}, & p = \infty \\
\end{cases}
\end{equation}
($p=1$ gives twice the total-variation distance, i.e.~$\|\mu-\nu \|_{1,\pi}=2\|\mu-\nu \|_{\TV}$).  The $\epsilon$-$L_{p}$\textbf{-mixing-time} of the chain is defined as
\begin{equation}
\label{eq: tau_p}
\tau_{p}(\epsilon):= \min \{t: \max_x \|\Pr_x^t-\pi \|_{p,\pi}\le \epsilon \}.
\end{equation}
When $\epsilon=1/2$ we omit it from the notation and terminology (i.e.~the $L_p$ mixing time is defined as the $1/2$-$L_p$ mixing time). 

\begin{definition}
\label{def: RI}
Let $G_i:=(V_i,E_i)$ ($i=1,2$) be two finite graphs. For $u,v \in V_i$, let $d_i(u,v)$ be graph distance (w.r.t.~$G_i$) between $u$ and $v$ (i.e.~the number of edges along the shortest path in $G_i$ between $u$ and $v$). We say that $f:V_1 \to V_2$ is a $K$-\emph{\textbf{rough isometry}} of $G_1$ and $G_2$ if 
\begin{itemize}
\item[(1)]
\[ \forall u,v \in V_1, \quad  (d_1(u,v)-1)/K \le d_2(f(u),f(v)) \le K(d_1(u,v)+1). \]

\item[(2)]For every $w \in V_2$, there exists some $v \in V_1$ such that $d_2(f(v),w) \le K $.
\end{itemize}
We say that $G_1 $ and $G_2$ are $K$-\emph{\textbf{roughly isometric}} if there exists such $f$ as above.
\end{definition}
We use the convention that $C,C',C_1,\ldots$ (resp.~$c,c',c_1,\ldots $) denote positive absolute constants which are sufficiently large (resp.~small). Different appearances of the same constant at different places may refer to different numeric values.
\subsection{The transitive case}
The following question is re-iterated from \cite{cf:Ding} (over there it is asked for the $L_1$ mixing time) and essentially also from \cite{cf:Pittet}.
\begin{question}
\label{q:tran}
Are transitive graphs robust?
\end{question}
Let $G=(V,E)$ be a finite connected transitive graph. Denote the eigenvalues of $I-P$ (where $P$ is the transition matrix of lazy simple random walk on $G$) by $0=\lambda_1<\lambda_2 \le \cdots \le \lambda_{|V|}$.
Since for transitive graphs the mixing time is the same from all initial states, averaging over the starting position yields that (cf.\ \cite[p.~284]{cf:Aldous}) for all $x$ and $t \ge 0$
$$\|\Pr_x^{2t}-\pi \|_{\infty,\pi}=\|\Pr_x^t-\pi \|_{2,\pi}^2 =\sum_{y} \pi(y)\|\Pr_y^t-\pi \|_{2,\pi}^2  = \sum_{i=2}^{|V|}\lambda_i^{2t} $$
(transitivity is used only in the middle equality).
Since  the eigenvalues of $I-P$ are robust (e.g.~\cite[Corollary 8.4]{cf:Aldous}), it follows that transitive graphs are robust under bounded perturbations which preserve transitivity.  A positive answer to Question \ref{q:tran} will be obtained by a positive answer to the following question. Is it the case that for  transitive graphs, also after a bounded perturbation, the ratio of the $L_2$ mixing time starting from the worst initial point with that starting from the best initial point is bounded? 
\subsection{Related work}
\label{s: related}
It is classical that under reversibility the $L_1$ mixing time can be characterized using hitting times of sets which are ``worst" in some sense (e.g.\ \cite[Ch.\ 24]{cf:LPW}). Thus in order to show that it is not robust, it suffices to construct an example in which hitting times are not robust. As we explain below, this is somewhat easier. 

Recently, the author and Peres obtained a characterization of the $L_2$ mixing time in terms of hitting time distributions. Namely, Theorem 1.1 in \cite{cf:LS} asserts that under reversibility the $L_2$ mixing time is within some universal constant from the minimal time $t$ such that for every set $A$ of stationary probability at most $1/2$, the probability that $A$ is not escaped from by time $t$ is at most $ \pi(A)+ \frac{1}{2} \sqrt{ \pi(A) (1-\pi(A))}$. While we do not use this characterization as part of our analysis of the example from the proof of Theorem \ref{thm: 1}, it guided us in its construction and in the choices of certain parameters. 

Ding and Peres \cite{cf:Ding} constructed a sequence of bounded degree graphs satisfying that the order of the total-variation mixing times strictly increases as a result of a certain sequence of bounded perturbations of the edge weights. Their construction was refined by the author and Peres in \cite{HP} (Theorem 3), which contains various additional results concerning sensitivity of mixing times and the cutoff phenomenon under small changes to the geometry of the chain. 

Our construction from the proof of Theorem \ref{thm: 1}  uses a key observation from \cite{cf:Ding}. Namely, that the harmonic measure of the walk on a tree may change drastically as a result of a bounded perturbation, and that this can be used to create examples in which hitting times are not robust in the following sense. The order of the expected hitting time of some large set $A$ which is ``worst" in some sense (both before and after the perturbation), starting from the worst initial state, may change as a result of the perturbation. The idea of exploiting the non-robustness of the harmonic measure was originally used by Benjamini \cite{cf:Benjamini} to study instability of the Liouville property.

Both in Ding and Peres' construction and in our construction the chain mixes rapidly (there in total-variation and here in $L_{2}$) once it reaches a certain ``huge" expander, $H$ (and because $H$ carries most of the stationary probability of the walk, the walk cannot mix before reaching it). We use the fact that the harmonic measure is sensitive (in the sense mentioned in the previous paragraph) in order to create ``shortcuts"  to $H$ for the walk on the original graph, $G_n$, which are essentially ``invisible" for the walk on the 2-stretched graph, $G_n'$. 

In order to change the mixing time in total-variation it suffices to change the order of the expected hitting time of $H$, starting from the worst initial state. However, this does not suffice in order to change the order of the $L_{\infty}$ mixing time. Thus we will work much harder and show that for every initial vertex $x$  the following holds. The probability that the walk on the original graph does not reach $H$ in some $t \ll \tau_1(G_n')$ steps (through one of the aforementioned shortcuts) is much smaller than  the $L_{\infty}$ distance from stationarity of the distribution of the walk at time $t$, conditioned to not reach $H$ by that time. In order to achieve this, we will plant more shortcuts to $H$ in regions in which the chain mixes slower. 

\medskip

In \cite{cf:Kozma} Kozma constructed a sequence of finite reversible Markov chains satisfying that $\rho^{(n)} \ge c \tau_{\infty}^{(n)} \log |\log (\min_x \pi_n (x) )| $,   for all $n$, where $\rho^{(n)}$ is the spectral profile bound (see \eqref{eq: SPbound}) on the $L_{\infty}$ mixing time of the $n$-th chain in the sequence, $\tau_{\infty}^{(n)} $. In Kozma's construction there are $ n-\lceil \log n \rceil +1 $ ``islands" $H_{\lceil \log n \rceil},\ldots,H_n$, each of size $2^{2^{n}}$. The weights are chosen so that the stationary distribution is the uniform distribution. In his example  the only way for the chain to ``escape" from some $H_i$ is by moving to a random state picked according to the uniform distribution on the state space, which occurs in each step w.p.~$n2^{-n} $.

\medskip As in Kozma's construction there is a weighted edge between all pairs of vertices, it is not clear that such a construction is possible in the bounded degree setup. Nevertheless, our example uses several ideas from Kozma's construction.  \begin{itemize}

\item In Kozma's construction, each $H_i$ has vertex set $A_i \times B_i$. The network on $H_i$ can be described as a ``Cartesian product" of the complete graphs on $A_i$ and $B_i$, in which the walk updates its $B_i$ co-ordinate at a rate $\alpha_i(n)=o(1)$.

 In our construction we use a bounded degree analog of the aforementioned network, which we denote by $U_i$. Namely, we replace the complete graphs by expanders. In order to delay the rate of transitions along the expander on the co-ordinate corresponding to $B_i$ in Kozma's construction, while keeping the graph unweighted, we stretch each of its edges by a factor $\ell_i(n) \to \infty $. 

\medskip

\item
In Kozma's example, for all $i$, $|A_i|=2^{2^{n}-2^{i}}$ and so $\pi(A_{i+1}) \approx \pi (A_i)^2 $. Moreover, each $A_i$ has roughly the same contribution to the spectral profile bound (see \eqref{eq: SPbound}). This is achieved by tuning the rates $\alpha_i(n)$ in an appropriate manner. Namely, by setting $\alpha_{i+1}(n) =2\alpha_i(n)  $, for all $i$. As noted in Remark \ref{rem: spectral}, in some sense, such behavior of the spectral profile is necessary in order for it to overshoot $\tau_{\infty} $ by an optimal factor (of order $\log \log |V|$). In our construction we will take $\ell_{i}^2(n)=2\ell_{i+1}^2(n)$ in order to obtain the same effect.    

\medskip

\item Recall that in Kozma's construction each pair of vertices are connected by a weighted edge, which has the effect of bringing the chain to stationary ``at once" at a fixed rate, and this is the only way the chain can escape from the ``island", $H_i$, it started at. In our construction we need to  somehow imitate this behavior (in a bounded degree, unweighted fashion). At the same time, as noted in Remark \ref{rem: spectral}, we need that after stretching some of the edges by a  factor of two, the walk must sequentially move through all of the islands in order to mix in total-variation. 

In order to achieve this behavior, we ``stitch" the ``islands" together so that the walk can escape from each of them either to a huge expander $H$ (once it is reached, the walk mixes rapidly) or to an island with adjacent index. The islands are glued to each other and to the expander $H$ in a way that allows us to manipulate (by stretching some edges) the probability of escaping an ``island" by reaching the expander $H$. 
\end{itemize}

\section{Preliminaries}
\label{s: Pre}
Recall that the \emph{spectral gap} of a reversible Markov chain with transition matrix $P$ on a finite state space is defined as the smallest non-zero eigenvalue of $I-P$. We say that a graph $G$ is a $\gl$-expander if the spectral gap of lazy simple random walk on $G$, denoted by $\gl(G)$, is at least $\gl$. We say that a sequence of graphs $(G_n)_{n \in I }$ is an \emph{expander family} if $\inf_{n \in I} \gl (G_n)>0 $. As mentioned in \S~\ref{s: related}, we shall use expanders as building blocks in our construction.
By abuse of terminology, below we often refer to a single graph as an ``expander". What we actually mean by that is that all of the graphs we refer to as expanders in the construction of the family of graphs we construct form together an expander family. 
\begin{definition}
\label{def: Cheeger}
Consider a Markov chain chain on a finite state space $\Omega $ with transition matrix $P$ and stationary distribution $\pi$. We define the {\em Cheeger constant} of the chain as  
\begin{equation*}
 \Phi:=\min_{A:0< \pi(A) \le 1/2}Q(A)/\pi(A), \quad \text{where} \quad Q(A):=\sum_{x \in A,y \notin A }\pi(x)P(x,y).
\end{equation*}
\end{definition}
The following  is the well-known discrete analog of Cheeger inequality. 
\begin{theorem}[e.g.~\cite{cf:LPW}, Theorem
13.14]
\label{thm: Cheeger}
If $P$ is reversible then 
\begin{equation}
\label{eq: Sinclair}
\Phi^2/2 \le \gl \le 2\Phi. 
\end{equation}
\end{theorem}
%
By \eqref{eq: Sinclair} a sequence of graphs $(G_n)_{n \in I }$ is an expander family iff $\inf_{n \in I } \Phi(G_n)>0$
.

The following proposition will be useful in what comes.
\begin{proposition}
\label{p: LS}
 There exists a constant $c_d> 0$ (depending only on $d$) such that if $H$ is a simple graph of maximal degree $d$ and
 $G$ is a $K$-stretch of $H$, then 
 \begin{equation}
 \label{eq: LS1}
 \Phi(G) \ge c_d \Phi(H)/K \text{ and so }\gl(G) \ge c_d^{2}\Phi^2(H)/(2K^2).
 \end{equation}  
\end{proposition}
\begin{proof}
The argument in \cite[Claim 2.2]{cf:LS} covers the case in which $|V(G)| \le \frac{3}{2}|V(H)|$.  We shall reduce the general case to this case.
It is easy to see that the minimum in the definition of $\Phi$ is always attained by a connected set. For $e=\{u,v\} \in E(H)$ let $\gamma_e$ be the collection of internal vertices along the  segment from $u$ to $v$ in $G$ which replaced the edge $e$ (if $e$ was not stretched, then $\gamma_e$ is empty). Let $\pi_G$ be the stationary distribution of the walk on $G$. Similar reasoning as in \cite[Claim 2.2]{cf:LS} shows that for some connected set $B$ 
\begin{itemize}
\item[(i)]  $ Q_G(B)/\pi_{G}(B) \le C_d \Phi(G)  $ (where $Q_G(B)$ denotes $Q$ w.r.t.~the graph $G$).
\item[(ii)] There is at most one $e\in E(H) $  so that $\gamma_e \setminus B$ an $\gamma_e \cap B $ are both non-empty. However, even if such $e$ exists, $\gamma_e \cap B $ is connected. 
\end{itemize}
The operation of contracting a set of vertices $D$ is defined as follows. Replace all of $D$ by a single vertex $x$ and for every edge $\{u,v\} $ with $u$ in $D$ and $v \notin D$, replace it with an edge $\{x,v\} $ and if also $v \in D$ replace it by a loop at $x$ of weight 2. Let $G_2$ be the network obtained by contracting each $\gamma_{e'}$, for all $e' \in E(H)$, apart from $\gamma_e$ for the $e$ so that $\gamma_e \setminus B$ an $\gamma_e \cap B $ are both non-empty (if such $e$ exists). If such $e$ exists, we contract $\gamma_e \cap B $ and also $(\gamma_e \setminus B) \cup \{v\} $, where $v$ is the endpoint of $e$ incident to $(\gamma_e \setminus B) $. 

Let $G_3$ be the graph obtained by deleting all loops from $G_2$. It is straightforward to check that $\Phi(G_2) \ge\Phi(G_3)/(2K) \ge c_d' \Phi(H)/K$ (where the second inequality follows from \cite[Claim 2.2]{cf:LS}). Conversely, by (i)-(ii),  $\Phi(G_2) \le Q_G(B)/\pi_{G}(B) \le C_d \Phi(G)  $ (where the first inequality is obtained by considering the set $B'$ in $G_2$ obtained from $B$ by replacing each $\gamma_{e'} \subset B $ by the corresponding vertex in $G_2$) and so indeed $\Phi(G) \ge c_d \Phi(H)/K$. 
\end{proof}
The Poincar\'e (spectral gap) inequality asserts that when time is scaled according to the inverse of the spectral gap, the $L_2$ distance from stationarity of every distribution decays exponentially in the number of (scaled) time units.
\begin{lemma}
\label{lem: Poincare}
Let $(\Omega,P,\pi)$ be a finite lazy  irreducible reversible Markov chain with spectral gap $\gl$.
Let
$\mu $ be a distribution on $\Omega$. Then
\begin{equation}
\label{eq: L2contraction}
 \|\Pr_\mu^t-\pi \|_{2,\pi} \le  e^{-\gl t}
\|\mu-\pi \|_{2,\pi}, \text{ for all }t \ge 0.
\end{equation}
\end{lemma}  
\begin{definition}
\label{def: lambdaA}
Let $(\Omega,P,\pi)$  be reversible. Let $A \varsubsetneq \Omega $. We define $\gl(A)$ (``the spectral gap of the set $A$") to be the smallest eigenvalue of the substochastic matrix obtained by restricting $I-P$ to $A$. Similarly, define $\Phi(A)= \min_{B \subset A }Q(B)/\pi(B) $. 
\end{definition}
The following extension of \eqref{eq: Sinclair} is due to Goel et al.~\cite[(1.4) and Lemma 2.4]{cf:Spectral}. For every irreducible reversible chain, and every set $A$ with $\pi(A) \le 1/2$ we have that
\begin{equation}
\label{eq: restrictedcheeger}
\Phi^2(A)/4 \le \gl(A) \le \Phi(A). 
\end{equation}

\medskip

The \emph{hitting time} of a set $D$ is defined as $T_D:=\inf \{t:X_t \in D\} $. Using the spectral decomposition of $P_A$, the restriction of $P$ to the set $A$ (e.g.~\cite[Lemma 3.8]{cf:Basu}), it is easy to show that for every set $A$ and $a,a' \in A$ we have that $P_A^t(a,a') \le \sqrt{\frac{\pi(a')}{\pi(a)}}e^{- \gl(A) t}$ and so
\begin{equation}
\label{eq: exitprob}
\forall a \in A, \quad \Pr_{a}[T_{\Omega \setminus A} > t]= \sum_{b \in A}P_A^t(a,b) \le |A| \max_{b \in A}\sqrt{\frac{\pi(b)}{\pi(a)}}e^{- \gl(A)t}.
\end{equation} 
Finally, we recall that the spectral profile upper bound on $\tau_{\infty} $ is \cite{cf:Spectral}
\begin{equation}
\label{eq: SPbound}
\rho:=8\lambda^{-1}\log 2 + \int_{\min_x \pi(x)}^{1/2}\frac{4dv}{v \Lambda(v) }, \quad \text{where} \quad \Lambda(v):=\inf_{A \subset \Omega :\pi(A) \le v}\lambda(A).
\end{equation} 
\section{Proof of Theorem \ref{thm: 1}}
\subsection{The construction}
For notational convenience we often omit ceiling signs. As described in \S~\ref{s: related} we shall construct graphs $U_n,\ldots,U_1$. For all $i$ we will have $$ 2^{2^{3n}+5n} \le |V(U_i)| \le C 2^{2^{3n}+6n}.$$ We will then ``stitch" them together to obtain the ultimate graph $G_n=(V_n,E_n)$. Hence
\begin{equation}
\label{eq: sizeGn}
n 2^{2^{3n}+5n} \le |V_n| \le Cn 2^{2^{3n}+6n}.
\end{equation}

For all $i$, the graph $U_i$ will be a Cartesian product of an expander $H_i=(V(H_i),E(H_i))$ with a graph $W_i $ (obtained by making a small modification to a certain tree $\cT_i $),
of sizes 
\begin{equation}
\label{eq: sizeHi}
2^{2^{3n}-2^{2n+i}} \le |V(H_i)| \le  2^{2^{3n}-2^{2n+i}+n-i}, \quad 2^{2^{2n+i}+5n} \le |V(W_i)| \le C  2^{2^{2n+i}+5n},
\end{equation}

\begin{itemize}
\item[Step 1.1]
Let $1 \le i \le n$. We now construct the tree $\cT_i$ which shall have roughly $2^{2^{2n+i}}$ \emph{good leafs} $\mathrm{GL}_i$ and roughly $2^{2^{2n+i-1}}$ \emph{bad leafs} $\mathrm{BL}_i$ (apart from $i=1$ which only has good leafs). It will be obtained by stretching the edges of a tree $\cT_{\mathrm{bs},i} $ ($\mathrm{bs}$ stands for ``before stretching") which is a ``binary tree", rooted at $o_i$, whose good leafs,  $\mathrm{GL}_i$,  are of depth $2^{2n+i}$, while its bad leafs,  $\mathrm{BL}_i$,  are all of some other depth, $j_i$. The sets of good leafs of the two trees $\cT_{\mathrm{bs},i} $ and $\cT_i $ are the same, and likewise for the sets of bad leafs. Note that usually a finite binary tree is defined so that all of its leafs are of the same depth, while here, crucially, the leaf set is partitioned into two sets of different depths. By abuse of terminology, we still refer to such a tree as a finite binary tree.  

\medskip
 
\item  We first describe the construction of $\cT_1 $ as it is simpler. Take a binary tree, $\cT_{\mathrm{bs},1}$, of depth $2^{2n+1}$ rooted at $o_1$. Denote its leafs by $\mathrm{GL}_1$. Then stretch each of its edges by a factor of $2^{5n}$.  

\medskip

\item We now construct $\cT_i$ for $1<i \le n $ in several steps. Before describing  $\cT_{\mathrm{bs},i} $ we consider a binary tree of depth $2^{2n+i-2}$, rooted at $o_i$, denoted by   $\cT_{\mathrm{fh},i} $  
($\mathrm{fh}$ stands for ``first half", as    $\cT_{\mathrm{fh},i} $  
  is the ``first half" of $\cT_{\mathrm{bs},i} $).  

\medskip

\item[Step 1.2]
For every vertex $u $ which is not a leaf of $\cT_{\mathrm{fh},i} $  
  we distinguish its two children by \emph{left} and \emph{right} child. For every vertex $u \in     \cT_{\mathrm{fh},i} $  
   let $\mathrm{Left}(u)$ (resp.~$\mathrm{Right}(u)$) be the number of left (resp.~right) children along the path from $o_{i}$ to $u$. Let 
\begin{equation}
\label{eq: g}
g(u)=\mathrm{Left}(u)-\mathrm{Right}(u).
\end{equation}
Denote the $k$-th level of a rooted tree $\cT$ by $\cL_k(\cT)$. We partition the leaf set of  $\cT_{\mathrm{fh},i} $  
   $\cL_{2^{2n+i-2}}(\cT_{\mathrm{fh},i}) $, into two parts: $\mathrm{GMP}_i$, the set of \emph{good middle points} and $\mathrm{BMP}_i$, the set of \emph{bad middle points} (they are ``middle points" w.r.t.~$\cT_{\mathrm{bs},i}$) defined as follows: 
$$\mathrm{GMP}_i:=\{u \in \cL_{2^{2n+i-2}}(\cT_{\mathrm{fh},i}) : g(u)\le 2^{2n+i-6} \} \text{ and }\mathrm{BMP}_i:= \cL_{2^{2n+i-2}}(\cT_{\mathrm{fh},i}) \setminus \mathrm{GMP}_i. $$
We now extend $\cT_{\mathrm{fh},i}$ so that the good leafs,   $\mathrm{GL}_i$, (resp.~bad leafs $\mathrm{BL}_i$), of the resulting tree, $\cT_{\mathrm{bs},i}$, will be the leafs which are decedents of $\mathrm{GMP}_i$ (resp.~$\mathrm{BMP}_i$).

\medskip

\item Attach to each vertex in $\mathrm{GMP}_i$ a binary tree of depth $2^{2n+i}-2^{2n+i-2}$. This makes the total number of leafs that have a vertex in $\mathrm{GMP}_i$ as an ancestor  $2^{2^{2n+i}}(1-o(1))$. 

\medskip
 
\item  Attach to each vertex in $\mathrm{BMP}_i$ a binary tree of depth $\lceil \log_2(2^{2^{2n+i-1}} /|\mathrm{BMP}_i| )\rceil $ so that the set of leafs that have a vertex in $\mathrm{BMP}_i$ as an ancestor, $\mathrm{BL}_i$, is of size 
\begin{equation}
\label{eq: sizeofBLi}
2^{2^{2n+i-1}} \le |\mathrm{BL}_i| \le 2^{2^{2n+i-1}+1} .
\end{equation}
Call the resulting tree $\cT_{\mathrm{bs},i}$.

\medskip

\item[Step 1.3] Denote the union of $\mathrm{GMP}_i$ with the collection of vertices of $\cT_{\mathrm{bs},i} $ which have an ancestor in $\mathrm{GMP}_i$ by $ \mathrm{GS}_i$ (a shorthand for ``good side"). 

\medskip

\item 

For each $1<i \le n$, stretch each edge of $\cT_{\mathrm{bs},i} $ that both of its end-points lie in  $ \mathrm{GS}_i$  by a factor of $2^{5n} $.
 Stretch each of the rest of the edges of $\cT_{\mathrm{bs},i} $ by a factor of
\begin{equation}
\label{eq: qi}
q_{i}:=\lceil 2^{8n-\frac{i}{2}}\rceil.
\end{equation}       
Call the resulting tree $\cT_i$.

\medskip

\item[Step 2] We now modify the tree $\cT_i $ in the region close to its leafs in order to make it an expander. We note that for $i=1$ the below modification is done only to the region close to $\mathrm{GL}_1$ as $\cT_1$ only has good leafs. Let ${\mathrm{GE}}_{i}=(V(\mathrm{GE}_{i}),E(\mathrm{GE}_{i})) $ and  ${\mathrm{BE}}_{i}=(V(\mathrm{BE}_{i}),E(\mathrm{BE}_{i}))$ ($\mathrm{E}$ stands for expander)  be 3-regular expanders of size $|\mathrm{GL}_i|/2$ and  $|\mathrm{BL}_i|/2$, resp.. Let $\mathrm{PGL}_i$ (resp.~ $\mathrm{PBL}_i$,) (a shorthand for ``parents of good (resp.~bad) leafs" w.r.t.~$\cT_{\mathrm{bs},i}$), be the collection of the $|\mathrm{GL}_i|/2$ (resp.~$|\mathrm{BL}_i|/2$) vertices of distance  $2^{5n}$ (resp.~$q_i$) w.r.t.~$\cT_i$ from $\mathrm{GL}_i$ (resp.~$\mathrm{BL}_i$). We naturally identify $\mathrm{PGL}_i$ (resp.~$\mathrm{PBL}_i$) with the collection of parents w.r.t.~$\cT_{\mathrm{bs},i} $ of the vertices in $\mathrm{GL}_i$ (resp.~$\mathrm{BL}_i$). Identify each vertex of  $\mathrm{PGL}_i $ (resp.~$\mathrm{PBL}_i$) with a vertex of $\mathrm{GE}_{i}$ (resp.~$\mathrm{BE}_i$) in a bijective manner.
Let $\cT_v=(V(\cT_v),E(\cT_v)) $ be the induced tree at $v$ w.r.t.~$\cT_i $ (i.e.~the induced tree on the set of vertices  that the path from them to $o_i$ goes through $v$).
For each $u,v \in \mathrm{PGL}_i$ (resp.~$\mathrm{PBL}_i$) such that $\{u,v\} \in E(\mathrm{GE}_{i})$ (resp.~$E(\mathrm{BE}_{i})$),  we connect each vertex $w \in V(\cT_u) $ to $\phi_{u,v}(w) \in V(\cT_v)$ by an edge, where $\phi_{u,v}$ is the trivial isomorphism of $\cT_u $ and $\cT_v $.
Call the resulting graph $W_i=(V(W_i),E(W_i))$.

\medskip

\item[Step 3] Let $1 \le i \le n-1$. Denote $s_i:=\prod_{j=i+1}^n |\mathrm{BL}_i|$. Note that by \eqref{eq: sizeofBLi}
$$2^{2^{3n}-2^{2n+i}} \le s_i \le   2^{2^{3n}-2^{2n+i}+n-i}. $$

Let $H_i=(V(H_i),E(H_i)) $ be a 3-regular expander of size $s_i $. Let $U_i=W_i \times H_i $ (a Cartesian product of $W_i$ and $H_i$). That is, $V(U_i)=V(W_i)\times V(H_i)$ and $\{ (u,h),(u',h') \} \in E(U_{i}) $ if either $u=u'$ and $\{h,h'\} \in E(H_i)$ or $h=h' $ and $\{u,u'\} \in E(W_i)$.

\medskip

We now ``stitch" together the graphs $U_n,\ldots,U_1$. We refer to  $$R_{i}:=\{o_i \} \times V(H_i) $$ (for $i=n$, $R_n:=\{o_n \}$) as the \emph{roots} of $U_i$. We refer to the set $$ \mathrm{Bad}_{i}:= \mathrm{BL}_{i} \times V(H_{i}) \quad  \text{(for }i=n,\, \mathrm{Bad}_{n}:=\mathrm{BL}_{n}) $$
 as the \emph{bad leafs} of $U_i $ (even though it is not a tree). Similarly, we refer to
$$\mathrm{Good}_{i}:= \mathrm{GL}_{i} \times V(H_{i}) $$
(for $i=n$, $\mathrm{Good}_{n}:=\mathrm{GL}_{n} $) as the \emph{good leafs} of $U_i $.

\medskip Note that $|\mathrm{Bad}_{i+1}|=|R_i |=s_{i}$, for all $i$. We shall connect $U_{i+1}$ to $U_i$ for all $1 \le i < n$, by identifying $\mathrm{Bad}_{i+1}$ with $R_i $ (step 4). We shall also connect all of the $U_i$'s ``at once" by connecting $\mathrm{Good}:=\cup_{i=1}^{n}\mathrm{Good}_{i} $ using one ``huge" expander $H$ of size $|\mathrm{Good}| $ (step 5). 

\medskip

\item[Step 4] For all $i<n$,  we take an arbitrary bijection $\phi: \mathrm{Bad}_{i+1} \to R_{i} $. We then replace each $u \in \mathrm{Bad}_{i+1}   $ and $\phi(u) $ with a new vertex, $u'$, that the set of edges which are incident to it is the union of the edges which are incident to $u$ (in $U_{i+1} $) and to $\phi(u)$ (in $U_i $). By abuse of notation we shall not distinguish between  the set $\{u':u \in \mathrm{Bad}_{i+1}  \}$ and the sets $\mathrm{Bad}_{i+1}$ and  $R_{i}$.

\medskip

\item[Step 5] We identify the set $\mathrm{Good} $ with a 3-regular expander $H=(V(H),E(H)) $ of size $|  \mathrm{Good} |  $ as follows. We label $\mathrm{Good}$ by the set $V(H)$ and connect $u,v \in \mathrm{Good} $ if $\{u,v \}\in E(H)$. Call the obtained graph $G_n=(V_n,E_n)$.

\medskip

\item[Step 6] We now describe $G_n'=(V_n',E_n')$. The only difference between it and $G_n $ is that for all $1<i \le n$ we perform the following step in between step 1.1 and step 1.2:

\medskip

\item For every non-leaf (w.r.t.~$\cT_{\mathrm{fh},i} $) vertex $u \in V(\cT_{\mathrm{fh},i}) $ and its right child $v$, we stretch the edge $\{u,v \}$ by a factor of $2$. Call the obtained graph $\cT_{\mathrm{fh},i}'$. The remaining steps are analogous to the ones in the construction of $G_n$. We spell them out for the sake of concreteness.

\medskip

\item

The leaf set of $\cT_{\mathrm{fh},i}'$ can be identified with that of $\cT_{\mathrm{fh},i}  $. Hence we may partition the leaf set of  $\cT_{\mathrm{fh},i}' $ into the same two sets  $\mathrm{GMP}_i$ and $\mathrm{BMP}_i$ defined in step 1.2 using the function $g$, taken again w.r.t.~the tree $\cT_{\mathrm{fh},i}$.  

\medskip

\item
After this is done, we can proceed with step 1.2 and obtain the tree $\cT_{\mathrm{bs},i}'$.

\medskip 

\item
We denote the tree of index $i$ obtained at the end of step 1.3 (when before step 1.2 one performs the aforementioned intermediate step) by $\cT_i'$. The only difference between $\cT_i' $ and $\cT_i$ is that the edges that were stretched by a factor $2$ in $\cT_{\mathrm{fh},i}' $ will be stretched ultimately by a total factor of $2q_{i}   $ rather than just $q_{i}$ (as in $\cT_i$). 

\medskip

\item The construction of $G_n'$ is concluded by completing the remaining steps of the construction of $G_n$ with $\cT_i'$ now playing the role of $\cT_i$ in construction of $G_n$.
%
\end{itemize}

\subsection{Analysis of the construction}
We start with the analysis of $G_n'$. We write $\Pr'$ and $\mathbb{E}'$ to denote probabilities and expectations w.r.t.~the walk on either $G_n' $, or  $\cT_i'$, for some $i$ (where the identity of the graph will be clear from context, and otherwise specified). Denote the stationary distribution of the walk on $G_n'$ by $\pi' $. We will show that (for all sufficiently large $n$) 
\begin{equation}
\label{eq: tau1Gn'}
\tau_{1}(G_n') \ge (n-1)2^{18n-2}=:\tau
\end{equation}

\medskip

Denote $\mathrm{Nice}_i=\mathrm{GMP}_i \times V(H_i) $  for $2 \le i \le n$ and $\mathrm{Nice}_1=R_1 $. Denote the connected component of $o_n $ w.r.t.~the cut $\mathrm{Nice}:=  \cup_{i=1}^n \mathrm{Nice}_i$ of the graph $G_n'$ by $\mathrm{Small}$. Recall that for every distribution $\mu $ on the state space we have that $\frac{1}{2} \|\mu-\pi '\|_{1,\pi'}=\|\mu-\pi' \|_{\TV}=\max_{A}\mu(A)-\pi'(A) $, where $\|\mu-\pi' \|_{\TV} $ is the total variation distance.  It follows that
$$ \|\Pr_{o_{n}}'(X_{\tau} \in \cdot ) - \pi '(\cdot) \|_{\TV} \ge \Pr_{o_{n}}'[X_{\tau} \in \mathrm{Small} ]- \pi'(\mathrm{Small} ) \ge \Pr_{o_{n}}'[T_{ \mathrm{Nice}}>\tau]- C_{0} 2^{-2^{n/2}},  $$
where we have used the following estimate in the last inequality
$$\pi'(\mathrm{Small} ) \le \max_{u,v} \frac{\deg (u)}{\deg(v)} \frac{|\mathrm{Small}|}{|V(G_n')|}  \le 15 |\mathrm{Small}|/|V(G_n')| \le C_{0} 2^{-2^{n/2}}.  $$ 
Hence, in order to prove \eqref{eq: tau1Gn'} it suffices to show that 
\begin{equation}
\label{eq: Easy}
\Pr_{o_{n}}'[T_{ \mathrm{Nice}}\ge \tau]=1-o(1).
\end{equation}
In order to establish \eqref{eq: Easy} we use the following lemma. 
\begin{lemma}
\label{lem:auxcalculations1}
Uniformly in $2 \le i \le n$ and $r \in R_i$, we have that 
\begin{equation}
\label{eq: hitbadbeforegood}
\Pr_{r}'[T_{\mathrm{Bad}_i } \le 2^{18 n-2}  \mid T_{\mathrm{Bad}_i}<T_{\mathrm{Nice}} ] \le \Pr_{o_i}'[T_{\mathrm{BL}_i} \le 2^{18 n-2} \mid T_{\mathrm{BL}_i} <T_{\mathrm{GMP}_i}  ]=o(1/n),
\end{equation}
\begin{equation}
\label{eq: hitprob}
\Pr_{r}'[T_{\mathrm{Nice}}<T_{\mathrm{Bad}_i} ] \le  \Pr_{o_i}'[T_{\mathrm{GMP}_i}<T_{\mathrm{BL}_i} ]+o(1/n)=o(1/n),
\end{equation}
where in the l.h.s of both \eqref{eq: hitbadbeforegood}-\eqref{eq: hitprob} the probability is taken w.r.t.~the walk on $G_n'$ and in the middle terms w.r.t.~the walk on $\cT_i'$. 
\end{lemma}
As $R_i=\mathrm{Bad}_{i+1}$ for all $1 \le i<n$, it follows from Lemma \ref{lem:auxcalculations1}   that w.p.~$1-o(1)$, started from $o_n$, before the walk on $G_n'$ reaches $\mathrm{Nice} $ it has to make its way from $R_{i+1}$ to $R_i $, for all $1 \le i <n$, and each of these $n-1$ ``stages" will take it at least $2^{18n-2} $ steps (in which case, it must be the case that $\tau_{1}(G_n') \ge (n-1)2^{18n-2}$, as desired). We believe that the assertion of Lemma \ref{lem:auxcalculations1} is intuitive and that from a high level perspective the proof is not complicated.  For the sake of completeness we choose to present a relatively detailed proof of Lemma \ref{lem:auxcalculations1}. 

\medskip

\emph{Proof of Lemma \ref{lem:auxcalculations1}:}  
The first inequalities in both \eqref{eq: hitbadbeforegood}-\eqref{eq: hitprob}  are obtained via a straightforward coupling argument and symmetry (of the sets $\mathrm{Nice}_i, \mathrm{Bad}_i$ and $R_i$ w.r.t.~the $H_i$ co-ordinate of the walk). The additive $o(1/n)$ term in the middle term of \eqref{eq: hitprob} is there to cover the following two scenarios w.r.t.~the walk on $G_n'$, started from $r \in R_i$: 
\begin{itemize}
\item  $T_{\cup_{j=i+1}^n \mathrm{Nice}_j}<T_{\mathrm{Nice}_i \cup \mathrm{Bad}_i}$. \item $T_{\mathrm{PBL}_i}<T_{\mathrm{Nice}}< T_{\mathrm{Bad}_i}  $.
\end{itemize}
Working out the details of the aforementioned couplings is left as an exercise.

\medskip

We now prove the equality in \eqref{eq: hitbadbeforegood}.
 Recall  the notation from step 6 of the construction. Let $(X_s^i)_{s \ge 0} $ be lazy simple random walk on $\cT_i' $.
We may view the walk  $(X_s^i)_{s \ge 0} $ only when it visits distinct vertices of $ \cT_{\mathrm{bs},i} $ (recall that $V(\cT_{\mathrm{bs},i})\subset V(\cT'_{\mathrm{bs},i}) \subset V(\cT_{i}') $). That is, consider the walk $Y_j=X_{S_j}^i$, for  $j \ge 0$, where $S_0=T_{V(\cT_{\mathrm{bs},i})} $ and for $j \ge 1$ $$S_{j}=\inf \{s> S_{j-1}: X_{s}^i \in V(\cT_{\mathrm{bs},i}) \setminus \{X_{S_{j-1}}^i \} \}. $$   For a set $D \subset V(\cT_{\mathrm{bs},i}) $ denote its hitting time w.r.t.~$(Y_s)$ by $$\tau_{D}:=\inf \{s: Y_{s} \in D \}. $$

Recall that $V(\cT_{\mathrm{bs},i})\setminus \mathrm{GS}_i$ (where $\mathrm{GS}_i$ is defined at step 1.3 of the construction) is the set of vertices of $\cT_{\mathrm{bs},i} $ which are connected in $\cT_i'$ to their neighbors w.r.t.~$\cT_{\mathrm{bs},i}$ by paths of length either $q_i$ or $2q_i$. It is easy to see that
\begin{itemize}
\item[(a)]
 If $Y_{j-1} \in V(\cT_{\mathrm{bs},i})\setminus \mathrm{GS}_i$, then the law of $S_{j}-S_{j-1} $ stochastically dominates the law of the   time it takes lazy simple random walk on $\Z$ to reach distance $q_i $ from its starting position (call this law $\bar \xi $). Moreover, started from $o_i$, conditioned on $\tau_{\mathrm{GMP}_i} \ge j$, we have that $S_j$ stochastically dominates $$S_j':=\sum_{k=1}^j \xi_k,$$ where $\xi_1,\xi_2,\ldots$ are i.i.d.~random variables with law $\bar \xi$.
\item[(b)] Started from $o_i$, conditioned on $\tau_{\mathrm{BL}_i}< \tau_{\mathrm{GMP}_i} $ we have that the law of $\tau_{\mathrm{BL}_i} $ stochastically dominates the law of the hitting time of the $ 2^{2n+i-2}$-th level of a binary tree, by simple random walk started from its root.
\item[(c)] The law of the last hitting time is highly concentrated around $3 \times 2^{2n+i-2}$. Hence 
$$\mathbb{P}_{o_i}[ \tau_{\mathrm{BL}_i }  \le b_i \mid \tau_{\mathrm{BL}_i}<\tau_{\mathrm{GMP}_i}     ] \le 2^{- c_0 2^{2n+i} }=o(1/n) , \quad \text{ where } b_i:=3 \times  2^{2n+i-3}. $$ 
\end{itemize}
 Note that (using \eqref{eq: qi}) $$2^{18n-2} \le \frac{2}{3}b_i q_i^2=\frac{2}{3}\mathbb{E}[S_{b_i}'] . $$ Using (a)-(c) above, it is not hard to verify that, for all $2<i \le n$, 
\begin{equation*}
\label{eq: tau11}
\begin{split}
& \Pr_{o_i}'[T_{\mathrm{BL}_i } \le 2^{18 n-2} \mid T_{\mathrm{BL}_i}<T_{\mathrm{GMP}_i} ] \\ & \le \Pr_{o_i}'[\tau_{\mathrm{BL}_i } \le b_{i} \mid \tau_{\mathrm{BL}_i}<\tau_{\mathrm{GMP}_i}   ]+\Pr_{o_i}'[S_{b_{i}} \le 2^{18 n-2} \mid \tau_{\mathrm{BL}_i}<\tau_{\mathrm{GMP}_i}  ]
 \\ & \le 2^{- c_0 2^{2n} } + \mathbb{P}[S_{b_i}'  \le \frac{2}{3}\mathbb{E}[S_{b_i}'] ]  
 =o(1/n).
\end{split}
\end{equation*}
where the first probability is taken w.r.t.~the walk on $\cT_i' $ and $\Pr_{o_i}'[S_{b_{i}} \le 2^{18 n-2} \mid \tau_{\mathrm{BL}_i}<\tau_{\mathrm{GMP}_i}  ] $ w.r.t.~$(X_s^i)_{s \ge 0}$ (in the sense that $(S_j)$ and the walk $(Y_j)$ are both determined by $(X_s^i)_{s \ge 0}$). This concludes the proof of \eqref{eq: hitbadbeforegood}. We now prove the equality in  \eqref{eq: hitprob}. 

\medskip

Fix $1<i \le n$. Consider an infinite tree $\cT' $ obtained by starting with $\cT $, an infinite binary tree, rooted at $o$, and stretching every edge between each vertex and its left (resp.~right) child by a factor of $q_i$ (resp.~$2q_i$). Let $v_k $ (resp.~$u_k$) be the last (resp.~first) vertex in $\cL_k(\cT)$ to be visited by a simple random walk on $\cT' $ started from $o$. 

\medskip

As before, for every $v \in V(\cT)$ let $g(v)=\mathrm{Left}(v)-\mathrm{Right}(v) $ (this function is defined w.r.t.~$\cT$). Using a network reduction and the well-known connection between random walks and electrical networks (cf.~\cite[Example 9.9]{cf:LPW} for a similar but simpler network reduction), we argue that 
\begin{lemma}
\label{lem: networkreduction}
$(g(v_{k+1})-g(v_k))_{k \ge 0} $ is a sequence of i.i.d.~random variables, each equals to 1 w.p.~$\frac{\sqrt{2}}{1+\sqrt{2}} $ and to $-1$ w.p.~$\frac{1}{1+\sqrt{2}}$. \end{lemma}

We do not care so much about the exact value of the constant $\frac{\sqrt{2}}{1+\sqrt{2}} $. The important point is that it is larger than $1/2$. Before proving the lemma, we explain how it implies \eqref{eq: hitprob}.   It follows from Lemma \ref{lem: networkreduction} that $g(v_{2^{2n+i-2}}) $ is highly concentrated around $2^{2n+i-2}(\frac{\sqrt{2}}{1+\sqrt{2}}) $. It is not hard to verify that this implies that $g(u_{2^{2n+i-2}}) \ge  2^{2n+i-5}  $ w.p.~$1-o(1/n)$ (uniformly in $i$). 

\medskip

Using a coupling argument  (in which the walks on $\cT'$ and on $\cT_i' $, started from $o$ and $o_i$, resp., are coupled so that they follow the same trajectory until they reach the $2^{2n+i-2} $-th level of the corresponding non-stretched trees) we obtain \eqref{eq: hitprob} (we leave the details as an exercise). This concludes the proof of $\tau_1(G_n') \ge (n-1)2^{18n-2} $.

\medskip

{\em Proof of Lemma \ref{lem: networkreduction}:} The fact that $(g(v_{k+1})-g(v_k))_{k \ge 0} $ is a sequence of i.i.d.~random variables follows by symmetry and the definition of the sequence $(v_i)_{i \ge 0}$.  Let $w$ be the effective conductance (e.g.~\cite[Ch.~9]{cf:LPW}) from the root of $\cT' $ to ``infinity" (i.e.~the limit of the effective conductance between the root and $\cL_\ell(\cT')$ as $\ell \to \infty$). Similarly, let $w_R$ and $w_L$ be the effective conductances between the root and infinity in the right and left subtrees of $\cT' $   (where the right subtree is obtained by deleting from $\cT' $ the left child of the root and all of its descendants, and the left subtree is similarly defined), resp.. Then $w=w_R+w_L$ and by symmetry $\frac{1}{w_R}=2q_i+\frac{1}{w}$ and $\frac{1}{w_L}=q_i+\frac{1}{w}$. The proof is concluded by solving these equations (which yields $w=\frac{1}{\sqrt{2}q_i}$ and $w_L=\frac{1}{(1+\sqrt{2})q_i} $) and noting that $\mathbb{P}[ g(v_{1})-g(v_0)=1]=\frac{w_L}{w}$, where the probability is taken w.r.t.~SRW on $\cT' $, started at its root.
\qed

\medskip

We now show that $\tau_{\infty}(G_n) \le C _{1}2^{18n}$. By standard results it suffices to show that there exist $C>0$ and $\ell>1$ such that for all $n$,  $$\tau_{\ell}(G_n) \le \lceil C 2^{18n} \rceil =:t$$
($C$ shall and $\ell$ shall be determine later, and throughout we assume that $C$ is sufficiently large and $\ell-1$ is sufficiently small so that all of the equations involving them below are satisfied). This follows from the following standard fact (e.g.~\cite[Lemma 2.4.6]{cf:SC}).
\begin{fact}
For every reversible Markov chain we have that for all $a>0$
\begin{equation}
\label{L2Linfty}
2 \tau_{2}(\sqrt{a})=\tau_{\infty}( a),
\end{equation}
$$\forall 1< \ell <2, \quad \tau_2(a^{m_{\ell}}) \le m_{\ell} \tau_{\ell}(a), \quad \text{where} \quad m_{\ell}:=1+\lceil (2-\ell)/(2\ell -2) \rceil.  $$
\end{fact}

\medskip

We argue that the spectral gap of $G_n$, denoted by $\gl(G_{n})$, satisfies
\begin{equation}
\label{eq: gap}
\gl(G_{n}) \ge  c_1 / (\max_i q_i )^{2}= c_1 q_2^{-2} \ge c_{2} 2^{-16n}.
\end{equation}

This follows from Proposition \ref{p: LS}, as $G_n $ is a $q_{2} $-stretch of a bounded degree expander.
 
\medskip

We start with an elementary observation which shall be used below repeatedly. Let $x \in V_n$ and $t>0$. Let $A$ be some event (which is determined by $(X_0,\ldots,X_t)$). Denote its complement by $A^c$. Then for all $1<\ell \le 2 $
\begin{equation}
\label{eq:totalprobfurmola}
\|\Pr_x^{t}- \pi \|_{\ell,\pi}= \Pr_x[A] \|\Pr_x^t[\cdot \mid A  ]- \pi(\cdot) \|_{\ell,\pi}+\Pr_x[A^{c}] \|\Pr_x^t[\cdot \mid A^{c}  ]- \pi(\cdot) \|_{\ell,\pi}.  \end{equation}
In particular, (using the fact that for every distribution $\mu$ on $V_n $ and every $\ell>1$ we have that $\|\mu- \pi \|_{\ell,\pi} \le C' |V_n|^{(\ell-1)/\ell} $) if $\Pr_x[A^{c}] \ll |V_n|^{-(\ell-1)/\ell}$, we may neglect the second term in the r.h.s. above, and concentrate on bounding $\Pr_x[A] \|\Pr_x^t[\cdot \mid A  ]- \pi(\cdot) \|_{\ell,\pi} $ from above.

\medskip

We now argue that starting from a vertex in the expander $H$ the walk mixes in $L_2$ in at most $2^{17n} $ steps (this is a wasteful estimate, but it suffices for our purposes). Hence, by the Markov property and the triangle inequality, it suffices to show that for $t=\lceil C 2^{18n} \rceil $, 
\begin{equation}
\label{eq: notinH}
\forall x, \quad \Pr_x[T_{V(H)}> t ]\|\Pr_x^t[\cdot \mid T_{V(H)}> t  ]- \pi(\cdot) \|_{\ell,\pi} \le 1/4.
\end{equation}
\begin{lemma}
\label{lem:fastmixH}
\[\lim_{n \to \infty}\max_{x \in V(H)}\| \Pr_x^{2^{17n}}- \pi \|_{2,\pi} =0.\] 
\end{lemma}
\emph{Proof:}
 Let $$J:=\mathrm{PGL}_n \cup (\cup_{i=1}^{n-1} \mathrm{PGL}_i \times V(H_i))$$
be the collection of vertices of distance (w.r.t.~$G_n$) $2^{5n}$ from $V(H)$. Each vertex in $J$ is connected to two vertices in $V(H)$ by a path of length $2^{5n}$.
Note that (started from $V(H)$) by time $2^{14n}$ the walk reaches $J$ w.p.~at least $1-C_{2}|V_{n}|^{-2}$. Hence (by the above discussion regarding \eqref{eq:totalprobfurmola}) we can neglect the case this fails. Let $i$ be the index such that $X_{T_{J}} \in V(U_i)$. Trivially, by the time the walk can cross a path  of length  $2^{5n}$ it must make at least $2^{5n} $ steps. But this means that, w.p.~at least $1-C_{2}|V_n|^{-2} $, the walk will make at least $ 2^{5n-6}$ steps along each of the expanders $H_i $ and $\mathrm{GE}_i$ between the last visit to $V(H)$ prior to $T_{J } $ and  $T_{J } $. Using \eqref{eq: sizeGn}-\eqref{eq: sizeHi} and \eqref{eq: L2contraction} it is not hard to show that this means that at time $2^{14n} $ the $L_2 $ distance of the walk (started from $H$) from $\pi$ is $O(n2^{6n})$ (as we later perform a similar more subtle calculation, we leave this as an exercise). Finally, the claim is obtained using \eqref{eq: gap} and the Poincar\'e (spectral gap) inequality \eqref{eq: L2contraction}. \qed

\medskip

To conclude the proof we now verify \eqref{eq: notinH}. The case that $x \in V(U_1)$ can be treated separately in a similar manner to the analysis below. The following lemma asserts that w.l.o.g.~we may assume that the initial position of the walk is in $R_i$ for some $i$. 
\begin{lemma}
\label{lem:startatroots}
Provided that $\ell-1$ is sufficiently small and $C$ is sufficiently large, 
for all $1<i \le n$ and $x \in V(U_{i}) $ (uniformly) \begin{equation}
\label{eq: didntreachR2}
\begin{split}
& \Pr_x[T_{V(H)}> t/2, T_{R_i \cup R_{i-1} }>t/2 ]  \|\Pr_x^{t/2}[\cdot \mid T_{V(H)}> t/2,  T_{R_i \cup R_{i-1} }>t/2 ]  - \pi(\cdot) \|_{\ell,\pi}=o(1) \\ 
  \end{split}
\end{equation}
\end{lemma}
\emph{Proof:}
We first argue that for all $1<i \le n$ and $x \in V(U_{i}) $ (provided that $C$ is taken to be sufficiently large)  
\begin{equation}
\label{eq: didntreachR1}
\Pr_x[T_{V(H)}> t/2, T_{R_i \cup R_{i-1} }>t/2 ] \le C_3 2^{-2^{2n+i}}.
\end{equation}
%
 Let $Z_i$ (resp.~$\bar Z_i $)
 be the number of steps the walk made by time $t/2$ along the $W_i $ (resp.~$H_i$) co-ordinate of $U_i$. As $\Pr_x[Z_i<t/30]<C_2|V(G_n)|^{-2}$, in order to prove \eqref{eq: didntreachR1} it suffices to show that (provided that $C$ is taken to be sufficiently large)
\begin{equation}
\label{eq: exittildeUi}
\max_{u \in V(W_i)}\Pr_u[T_{\mathrm{GL}_i \cup \mathrm{BL}_i \cup \{o_i\}}>t/30] \le  2^{-2^{2n+i}} 
\end{equation}
where the probability is taken w.r.t.~the walk on $W_i $. 
  Let $\tilde P_i $ and $\tilde \pi_i $ be the transition matrix and stationary distribution (resp.) of lazy simple random walk on $W_i $.
  Let $\tilde \gl_{i} $ be the smallest eigenvalue of the substochastic matrix obtained by restricting $I- \tilde P_{i}$ to $V(W_i) \setminus (\mathrm{GL}_i \cup \mathrm{BL}_i \cup \{o_i\} ) $. Using \eqref{eq: restrictedcheeger} and an obvious extension of Proposition \ref{p: LS} it is not hard to verify that
$$ \tilde \gl_{i} \ge \tilde c 2^{-(16n-i)}.  $$ 
The proof of \eqref{eq: exittildeUi} (and so also of \eqref{eq: didntreachR1}) can now be concluded using \eqref{eq: exitprob}.  

\medskip

Our next goal is to show that
\begin{equation}
\label{eq:goal2} \| \Pr_x[X_{t/2} \in \cdot \mid T_{V(H)}> t/2, T_{R_i \cup R_{i-1} }>t/2 ]- \pi(\cdot)\|_{\ell,\pi} \le C_7 2^{(\ell - 1)2^{2n+i+2}} \end{equation}
Observe that this, in conjunction with \eqref{eq: didntreachR1}, implies the assertion of the lemma.

\medskip

 For every $(u,h) \in V( W_{i}) \times V(H_i) $ denote
$$p_{u,h}:=\Pr_x[X_{t/2} \in (u,h)  \mid T_{V(H)}> t/2, T_{R_i \cup R_{i-1} }>t/2 ]  $$
As before, since $\Pr_x[\bar Z_i<t/20]<C_2|V(G_n)|^{-2} $ we can neglect the case that $\bar Z_i<t/20 $. Hence we can make the following estimate: for every $h \in V(H_i)$ we have that for all $1<i \le n$ and $x \in V(U_{i}) $ 
\[\sum_{u \in V(W_i) } p_{u,h} \le C_4' \sum_{u \in V(W_i)}\Pr_x[X_{t/2} \in (u,h)  \mid T_{V(H)}> t/2, T_{R_i \cup R_{i-1} }>t/2 , \bar Z_i \ge t/20 ]. \]
Note that the conditioning on $T_{V(H)}> t/2, T_{R_i \cup R_{i-1} }>t/2$, does not affect the projection of the walk onto its $H_i$ co-ordinate, given the number of steps made on that co-ordinate. Also observe that conditioned on $T_{V(H)}> t/2, T_{R_i \cup R_{i-1} }>t/2$,
 the aforementioned projection (up to time $t/2$, viewed in times in which the $H_i$ co-ordinate changes) is itself a random walk on the expander $H_i$. This, in conjunction with the fact  that the $L_{\infty}$ mixing time of SRW on $H_i$ is at most $t/20$ (to see this, use \eqref{eq: L2contraction}, \eqref{L2Linfty} and the fact that  $|V(H_i)|=|R_i| \le     2^{2^{3n}-2^{2n+i}+n-i} $), implies that\[\sum_{u \in V(W_i)}\Pr_x[X_{t/2} \in (u,h)  \mid T_{V(H)}> t/2, T_{R_i \cup R_{i-1} }>t/2 , \bar Z_i \ge t/20 ] \le \bar C_4/|V(H_i)|. \]
  
Combining the last two inequalities, we get that for every $h \in V(H_i)$,  for all $1<i \le n$ and $x \in V(U_{i}) $ 
\begin{equation}
\label{eq: mixinHi}
\sum_{u \in V(W_i) } p_{u,h}\le C_4/|V(H_i)|. 
\end{equation}
Since $\sum_{i}^{k}x_i^{\ell}$ subject to the constraints $  x_i \in \R_+  $, for all $i$, and $\sum_{i=1}^k x_i=a >0$ is at most $a^{\ell}$, provided that $\ell \ge 1$, using \eqref{eq: mixinHi} we get that for all $1<i \le n$ and $x \in V(U_{i}) $ 
\begin{equation}
\label{eq: hard}
\begin{split}
& \| \Pr_x[X_{t/2} \in \cdot \mid T_{V(H)}> t/2, T_{R_i \cup R_{i-1} }>t/2 ]- \pi(\cdot)\|_{\ell,\pi}^{\ell} \\ & \le C_5 \sum_{h \in V(H_i)} 1/|V(G_n)| \sum_{u \in V(W_i)} p_{u,h}^{\ell}/(1/|V(G_n)| )^{\ell} \\ & \le C_6 |V(G_n)|^{\ell -1} \sum_{h \in V(H_i)}|V(H_i)|^{-\ell}=C_6 \left( \frac{|V(G_n)|}{|V(H_i)|} \right)^{\ell -1}\le C_7 2^{(\ell - 1)2^{2n+i+2}}, 
\end{split}
\end{equation}
as desired.
 \qed 

Using \eqref{eq: didntreachR2}, in order to prove \eqref{eq: notinH} it suffices to show that for all $i$ and all $x \in R_i$ we have (provided that  $\ell-1$ is sufficiently small and $C$ is sufficiently large) that

\begin{equation}
\label{eq: notinH2}
\Pr_x[T_{V(H)}> t/2 ]\|\Pr_x^{t/2}[\cdot \mid T_{V(H)}> t/2  ]- \pi(\cdot) \|_{\ell,\pi} \le 1/8.
\end{equation}

\medskip
The case $i=1$ can be treated separately in a similar manner to the analysis below. Hence below we assume that $i>1$ and $x \in R_i$.

We argue that (provided that $C$ is sufficiently large) for all  $i>1$ and $x \in R_i$,

\begin{equation}
\label{eq: smallY1}
c_{3}\Pr_x[T_{V(H)}> \frac{t}{2} ] \le  \Pr_x[T_{R_{i-1}}<T_{V(H)} ] \le  \frac{|\mathrm{BMP_i} |}{|\mathrm{GMP_i} \cup \mathrm{BMP_i}| } = \frac{|\mathrm{BMP_i} |}{2^{2^{2n+i-2}}} \le 2^{-c_2 2^{2n+i} }.  
\end{equation}
The proof of the first inequality in \eqref{eq: smallY1} is similar to the proof of \eqref{eq: didntreachR1} and hence omitted. The last inequality in \eqref{eq: smallY1} is trivial. Finally, using a coupling argument, the second inequality in \eqref{eq: smallY1} can be proven by noting that started from $o_i$, the hitting distribution of $\mathrm{GMP_i} \cup \mathrm{BMP_i} = \cL_{2^{2n+i-2} \times q_{i}}(\cT_i)=\cL_{2^{2n+i-2}}( \cT_{\mathrm{bs},i}) $ w.r.t.~simple random walk on $\cT_i $ is the uniform distribution on this set. We omit the details.  

\medskip

It is easy to verify that starting at $x \in R_i$,  the walk will make by time $t/2$ at least $2^{5n} $ consecutive steps on either $U_i$ or $U_{i+1}$ w.p.~at least $1-C_{8}|V(G_n)|^{-2} $. On this event, as before, the walk will make by time $t/2$ at least $2^{5n-6} $ steps on either $H_{i}$ or on both of $H_{i+1}$ and $\mathrm{BE}_{i+1} $, (in one of its visits to either $U_i$ or $U_{i+1}$, resp., in which it stays in it for at least $2^{5n}$ steps) w.p.~at least $1-C_{9}|V(G_n)|^{-2} $. Using the same reasoning as in \eqref{eq: mixinHi}-\eqref{eq: hard}, we get that for all  $i \ge 1$ and $x \in R_i$,
 \begin{equation}
\label{eq: smallY2}
 \|\Pr_x^{t/2}[\cdot \mid T_{V(H)}> t/2]  - \pi(\cdot) \|_{\ell,\pi}^{\ell} \le  C_7 2^{(\ell - 1)2^{2n+i+2}}.
\end{equation}
We leave the verification of \eqref{eq: smallY2} as an exercise. Using \eqref{eq: smallY1}-\eqref{eq: smallY2} we get that for all  $i > 1$ and $x \in R_i$ (provided that $\ell-1$ is sufficiently small and $C$ is sufficiently large),

\begin{equation}
\label{eq: largeY1}
\Pr_x[T_{V(H)}> t/2 ]\|\Pr_x^{t/2}[\cdot \mid T_{V(H)}> t/2]  - \pi(\cdot) \|_{\ell,\pi} \le 1/8.
\end{equation}
This concludes the proof of $\tau_2(G_n) \le \lceil C 2^{18 n} \rceil$. \qed \section{Proof of Proposition \ref{prop: robustness}}
(1) Torus $\Z_n^d$: (the argument is almost identical to that in \cite{cf:Ding}) The lower bound follows from the fact that the inverse of the spectral gap (which is robust and up to a $\log 2$ factor, bounds $\tau_1$ from below)  is at least $c_d n^{2}$, while the upper bound can be deduced from the Morris and Peres' Evolving sets bound \cite{cf:Evolving}, which is also robust. (2) The giant component of an Erd\H os-R\'enyi supercritical random graph $\mathcal{G}(n,c/n) $: the existence of paths of length $\Theta (\log n)$ implies that the inverse of the Log-Sobolev constant, is at least $c \log^3 n $ (this can be derived using Lemma 4.2 in \cite{cf:Spectral}). This provides a robust lower bound on $\tau_{\infty}$. The upper bound follows from the fact that in this case the spectral gap satisfies $\lambda:=\Theta(1/ \log^2 n)$ and so $\tau_{\infty} \le C \lambda^{-1}\log n$ (since $c$ is fixed, the average degree is uniformly bounded). Finally, for example (3) and (4) the argument is identical to that in \cite{cf:Ding} and hence omitted.
\section*{Acknowledgements}
The author is grateful to Gady Kozma and Yuval Peres for useful discussions. The author would like to thank them and the anonymous referee for making suggestions which improved the presentation of the note.

\end{document}